\documentclass{amsart}
\usepackage{graphicx} 
\usepackage[T1]{fontenc}
\usepackage{amsmath,amssymb,amsthm,xcolor,mathrsfs,enumitem}
\newcommand{\intR}{\int_{\R^N}}

\newcommand{\F}{\mathcal{F}}

\newcommand{\R}{\mathbb{R}}
\newcommand{\C}{\mathbb{C}}
\newcommand{\N}{\mathbb{N}}

\newcommand{\ptl}{{\partial}}
\providecommand{\abs}[1]{\lvert#1 \rvert}
\newcommand{\wh}{\widehat }
\newcommand{\K}{\mathcal K}

\newcommand{\U}{\mathcal U}
\newcommand{\RR}{\mathcal R}

\newcommand{\E}{\mathcal E}

\newcommand{\PP}{\mathcal P }
\DeclareMathAlphabet{\mathpzc}{OT1}{pzc}{m}{it}
\newcommand{\w}{\mathpzc w  }
\newcommand{\M}{\mathcal M }

\newcommand{\loc}{\operatorname{loc}}
\providecommand{\norm}[1]{\lVert#1 \rVert}
\renewcommand{\div}{\operatorname{div}}
\renewcommand{\Re}{\operatorname{Re}}

\newcommand{\ve}{\varepsilon }

\numberwithin{equation}{section}
\newtheorem{teo}{Theorem}[section]

\newtheorem{prop}[teo]{Proposition}
\newtheorem{lema}[teo]{Lemma}

\newtheorem{cor}[teo]{Corollary}

\newcommand{\bqq}{\begin{equation*}}
\newcommand{\eqq}{\end{equation*}}
\newcommand{\bq}{\begin{equation}}
\newcommand{\eq}{\end{equation}}
\newcommand{\grad}{\nabla}
\makeatletter

\newenvironment{ecu}{\begin{equation}\left\lbrace\begin{aligned}}{\end{aligned}\right.\end{equation}\@ignoretrue}
\makeatother
\newcommand{\sublim}{\operatornamewithlimits{\longrightarrow}}   
\begin{document}
\title[Nonexistence of traveling waves for a nonlocal GP equation]{Nonexistence of traveling waves for a nonlocal Gross--Pitaevskii equation}                                 
    \author{Andr\'e de Laire}                                
    \address{Laboratoire Jacques-Louis Lions\\
Universit\'e Pierre et Marie Curie\\
 Bo\^ite Courier 187\\
 75252 Paris Cedex 05, France}                                    
    \email{delaire@ann.jussieu.fr}                                      
    \thanks{The author is grateful to F. B\'ethuel  for interesting and helpful discussions
and to P. Gravejat for his collaboration in proving Proposition~\ref{lema-tecnico}}                                     
    \begin{abstract}
We consider a 
 Gross--Pitaevskii equation with a nonlocal interaction potential.
 We provide sufficient conditions on the potential such that there exists a range of speeds
 in which nontrivial traveling waves do not exist.
\end{abstract}   
\maketitle
\section{Introduction}\label{intro}
\subsection{The problem}
We consider finite energy traveling waves for the 
nonlocal Gross--Pitaevskii equation
\begin{equation} \label{NGP}
i \ptl _t u-\Delta u-u(W*(1-\abs{u} ^2))=0,  \quad \ u(x,t)\in \C, \ x\in \R^N, \ t\in  \R.
\end{equation}
Here $*$ denotes the
convolution in $\R^N$ and $W$ is a real-valued even distribution. The aim of this work is
to provide sufficient  conditions on the potential $W$ such that these traveling waves are necessarily constant
for a certain range of speeds.  Equation  \eqref{NGP}
is Hamiltonian and its energy
\begin{equation*}\label{energy}
E(u(t))=\frac12 \int_{\R^N}\abs{\nabla u(t)}^2\,dx +\frac 14 \int_{\R^N}(W*(1-\abs{u(t)}^2))(1-\abs{u(t)}^2)\,dx
\end{equation*}
is formally conserved. A traveling wave of speed $c$ that 
propagates along the $x_1$-axis is a solution of the form
$$u_c(x,t)=v(x_1-ct,x_\perp), \quad {x_\perp=(x_2,\dots, x_N).}$$
Hence the profile $v$ satisfies
\begin{equation}\label{TW}\tag{NTW$c$}
ic\ptl_1v+\Delta v+v(W*(1-\abs v ^2))=0  \textrm{ in } \R^N
\end{equation}
and by using complex conjugation, we can restrict us to the case $c\geq 0$.
Note that any constant (complex-valued) function $v$ of modulus one
verifies \eqref{TW}, so that we refer to them as the trivial solutions.

Notice that, in the case that $W$ coincides with the Dirac delta function,  \eqref{TW} reduces to
the classical Gross--Pitaevskii equation 
\begin{equation}\label{TW-local}\tag{TW$c$}
ic\ptl_1v+\Delta v+v(1-\abs v ^2)=0  \textrm{ in } \R^N.
\end{equation}
Equation \eqref{TW-local} has been intensively studied in the last years.
We refer to \cite{bethuel0} for a survey.
From now on we suppose that $N\geq 2$ and  we recall the following results.

\begin{teo}[\cite{brezis,bethuel2,gravejat,gravejat-n-2}]\label{teo:brezis}
Let $v\in H_{\loc}^1(\R^N)$
be a finite energy solution  of \eqref{TW-local}.
Assume that one of the following cases hold
\begin{enumerate}[label=\textup{(}{\roman*}\textup{)},ref=({\roman*})]
 \item\label{caso-c-0}  $c=0$.
\item\label{teo:gravejat}  $c> \sqrt 2$.
 \item $N= 2$ and $c= \sqrt 2$.
\end{enumerate}
 Then $v$ is a constant function of modulus one.
 \end{teo}

\begin{teo}[\cite{bethuel2,bethuel-orlandi,chiron,bethuel,maris}]\label{teo:maris}
There is some nonempty set $A\subset (0,\sqrt 2)$ such that for all $c\in A$ there exists a nonconstant
 finite energy solution of \eqref{TW-local}. 
Furthermore, assume that  $N\geq 3$.  Then there exists a nonconstant finite energy 
 solution of \eqref{TW-local} for all $0<c< \sqrt 2$.
 \end{teo}

It would be reasonable to expect to generalize in some way these theorems
to the nonlocal equation \eqref{TW}. The aim of this paper is
to investigate the analogue of  Theorem~\ref{teo:brezis} in the cases \ref{caso-c-0} and \ref{teo:gravejat}.  Before stating our precise results,
we give some  motivation about the critical speed.

\subsection{Physical motivation}
As explained in \cite{delaire}, \eqref{NGP}
can be considered as a generalization of the equation
\begin{equation}
i\hbar\partial_t\Psi(x,t)=\frac{\hbar^2}{2m}\Delta\Psi(x,t)+\Psi(x,t) \int_{\R^N} \abs{\Psi(y,t)}^2V(x-y)\,dy, 
\text{ in $\R^N\times \R,$}
\label{GP-full}
\end{equation}
introduced by Gross \cite{gross} and 
Pitaevskii \cite{pitaevskii} to describe the kinetic of a weakly interacting Bose gas
of bosons of mass $m$, where $\Psi$ is the wavefunction governing the condensate 
in the Hartree approximation
 and $V$ describes the interaction between bosons.

In the most typical approximation, $V$ is considered
as a Dirac delta function. Then this model has  applications  in several  areas of physics, such as
superfluidity, nonlinear optics and Bose--Einstein condensation \cite{JPR1,JPR2,kivshar,coste}.
It seems then natural to analyze equation \eqref{GP-full} for more
general interactions. Indeed, in the study of superfluidity, supersolids
and Bose--Einstein condensation, different types of nonlocal potentials have been proposed 
\cite{bogo,berloff0,deconinck,kraenkel,rica0,rica,yi,cuevas,remi,aftalion}.

Let us now proceed formally and consider a constant function $u_0$ of modulus one.
Since \eqref{NGP} is invariant by a change of phase, we can assume $u_0=1$.
Then the linearized equation of \eqref{NGP} at $u_0$ is given by
\bq\label{linear}
i \ptl_t \tilde u -\Delta \tilde u +2 W*\Re( \tilde u)=0.
\eq
Writing $\tilde u=\tilde u_1+i\tilde u_2$  and taking real and imaginary parts in \eqref{linear}, 
we get 
\begin{align*}
-\ptl_t \tilde u_2 -\Delta \tilde u_1+2 W*\tilde u_1=0,\\
\ptl_t \tilde u_1-\Delta \tilde u_2=0,
\end{align*}
from where we deduce that 
\bq\label{linear2}
\ptl^2_{tt} \tilde u-2 W*(\Delta \tilde u)+\Delta^2 \tilde u=0.
\eq
By imposing $\tilde u=e^{i(\xi.x-wt)}$, $w\in \R$, $\xi\in \R^N$, as a solution of \eqref{linear2},
we obtain the dispersion relation
\bq\label{bogo}
(w(\xi))^2=\abs{\xi}^4+2\widehat W(\xi)\abs{\xi}^2,
\eq
where $\wh W$ denotes the Fourier transform of $W$.
Supposing that $\wh W$ is positive and continuous at the origin, 
we get  in the long wave regime, i.e. $\xi\sim 0$,
\bqq
w(\xi)\sim (2\widehat W(0))^{1/2}\abs{\xi}.
\eqq
Consequently, in this regime we can identify $(2\widehat W(0))^{1/2}$ as the speed of sound waves (also called sonic speed),
so that we set 
\bqq
c_s(W)=(2\widehat W(0))^{1/2}.
\eqq

The dispersion relation  \eqref{bogo} was first observed by Bogoliubov \cite{bogo}
on the study of Bose--Einstein gas and under some physical considerations he established
that the gas should move with a speed less than $c_s(W)$ to preserve its superfluid 
properties. From a mathematical point of view and comparing with Theorems~\ref{teo:brezis} and \ref{teo:maris}, 
this encourages us to think that the nonexistence of a nontrivial 
 solution of \eqref{TW} is related to the condition
\bq\label{lim-W} c> c_s(W).\eq
Actually, in Subsection \ref{statement} we provide results in this direction  and in
Subsection~\ref{examples} we specify the discussion for some explicit potentials $W$
which are physically relevant.

\subsection{Hypotheses on  $W$}
Let us introduce the spaces 
$\M_{p,q}(\R^N)$ of tempered distributions $W$ such that the linear operator
$f\mapsto W*f$ is bounded from $L^p(\R^N)$ to  $L^q(\R^N)$. We will use the following 
hypotheses on $W$. 
\begin{enumerate}[label=({H\arabic*}),ref=\textup{({H\arabic*})}]
\item\label{H0} $W$ is a real-valued even temperated distribution.
  \item\label{H-infty} $W\in \M_{2,2}(\R^N)$. Moreover, if $N \geq 4$, 
\bq\label{hyp-high}
 W\in \M_{N/(N-1),\infty}(\R^N)\cap \M_{2N/(N-2),\infty}(\R^N)\cap \M_{2N/(N-2),2N/(N-2)}(\R^N).
\eq
 \item \label{H-resto}  $\wh W$ is differentiable a.e. on $\R^N$ and for all $j,k\in\{1,\dots,N\}$ the map 
 $\xi \to \xi_j\ptl_k\wh W(\xi)$ is bounded and continuous a.e. on $\R^N$.
\item \label{W-positivo} $\wh W \geq 0$ a.e. on $\R^N$.
\item \label{H-reg} $\wh W$ is of class $C^2$ in a neighborhood of the origin and $\wh W(0)>0.$
 \end{enumerate}
Recall that the condition $W\in \M_{2,2}(\R^N)$ is equivalent to  $\wh W\in L^\infty(\R^N)$ (see e.g. \cite{grafakos}).
Therefore \ref{W-positivo} makes sense provided that \ref{H-infty} holds. It is proved in  \cite{delaire} that under the assumptions
\ref{H0}, \ref{H-infty} and \ref{W-positivo} the Cauchy problem for \eqref{NGP} with
nonzero condition at infinity is globally well-posed. Actually,  condition \eqref{hyp-high}
is more restrictive that the one used in \cite{delaire} in dimension $N\geq 4$, but we need it to ensure the regularity
of solutions. More precisely, in Section~\ref{section-regularity} we prove that under the hypothesis \ref{H-infty}, the 
solutions of \eqref{TW} are smooth and satisfy
$$\abs{v(x)}\to 1, \ \grad v(x)\to 0, \qquad \text{ as }\abs{x}\to\infty.$$
On the other hand, by Lemma~\ref{lema-reg2},  \eqref{hyp-high}
is at least fulfilled for $W\in L^1(\R^N)\cap L^N(\R^N)$.

Assumption \ref{H-infty} also implies that $E(v)$ is finite in the energy space
$$\E(\R^N)=\{ \varphi \in H^1_{\loc}(\R^N) :  1-\abs{\varphi}^2 \in L^2(\R^N), \grad \varphi \in L^2(\R^N)\}.$$
Furthermore, if \ref{W-positivo} also holds, then by the Plancherel identity 
$$E(v)=  \frac12 \int_{\R^N}\abs{\nabla v}^2+\frac 1{4(2\pi)^{N}} \int_{\R^N}\wh W\,\abs{\widehat{1- \abs{v}^2}}^2 \geq 0.$$
In Subsection \ref{examples} we show several examples of distributions $W$
satisfying the conditions \ref{H0}--\ref{H-reg}.

\subsection{Statement of the results}\label{statement}

\begin{teo}\label{teo1}
Assume that  $W$ satisfies  \ref{H0}--\ref{H-reg}. 
Let $c>c_s(W)$ and suppose  that there exist constants  $\sigma_1,\dots,\sigma_N\in \R$ such that
\bq
\label{sigma-1-1} \wh W(\xi)+\alpha_c \sum_{k=2}^N\sigma_k \xi_k\ptl_k \wh W(\xi)-\sigma_1 \xi_1\ptl_1\wh W(\xi)\geq  0, \text{ for a.a. } \xi \in \R^N, 
\eq
and 
\bq
\label{sigma-2-2} \sum_{k=2}^N \sigma_k +\min\left\{-\sigma_1-1,\frac{\sigma_1-1}{\alpha_c+2},2\alpha_c \sigma_j+\sigma_1-1\right\}\geq 0,  
\eq
for all $j\in \{2,\dots, N\}$, where $\alpha_{c}:={c^2}/(c_s(W))^2-1$.
Then nontrivial solutions of \eqref{TW} in $\E(\R^N)$ do not exist. 
\end{teo}
To apply Theorem~\ref{teo1} we need to verify the existence of the constants
$\sigma_1,\dots,\sigma_N$ satisfying  \eqref{sigma-1-1} and \eqref{sigma-2-2}.
To avoid this task,  we provide two corollaries 
where the conditions for the nonexistence of traveling waves are expressed 
only in terms of $W$.
\begin{cor}\label{cor-teo1}
 Assume that  $W$ satisfies  \ref{H0}--\ref{H-reg} and also that
\bq\label{hyp-cor}
 \wh W(\xi)\geq \max\left\{1,\frac{2}{N-1}\right\} \sum_{k=2}^N\abs{ \xi_k\ptl_k \wh W(\xi)}+  \abs{\xi_1\ptl_1\wh W(\xi)},  \text{ for a.a. } \xi \in \R^N.
\eq
Suppose that $c>c_s(W)$. Then nontrivial solutions of \eqref{TW} in $\E(\R^N)$ do not exist.  
\end{cor}

\begin{cor}\label{cor-teo}
 Assume that  $W$ satisfies \ref{H0}--\ref{H-reg}. Suppose that  
\bq\label{ine:cor}
c_s(W)<c \leq c_s(W)\left( 1+\inf_{\xi\in \R^N} \frac{(N-1)\wh W(\xi)}{\sum_{k=2}^N \abs{\xi_k\ptl_k \wh W(\xi)}}\right)^{1/2}, \text{ for a.a. } \xi \in \R^N.
\eq
Then nontrivial solutions of \eqref{TW} in $\E(\R^N)$ do not exist.  
\end{cor}
Concerning the static waves, we have the following result. 
\begin{teo}\label{teo-c-0}
Assume that  $W$ satisfies  \ref{H0}--\ref{W-positivo}.  Suppose that  $c=0$ and that
\bq\label{der-c-0}
\xi_j\ptl_j \wh W(\xi)\leq 0, \text{ for a.a. } \xi \in \R^N,
\eq
for all $j\in \{1,\dots, N\}.$
Then nontrivial solutions of \eqref{TW} in $\E(\R^N)$ do not exist.  
\end{teo}
Note that in the case $W=a\delta$, $a>0$, $\wh W=a$ and so that $\grad \wh W=0$. Then
conditions \eqref{hyp-cor}, \eqref{ine:cor} and \eqref{der-c-0} hold. Therefore, invoking Corollary~\ref{cor-teo1} or
\ref{cor-teo} and Theorem~\ref{teo-c-0}  we obtain the nonexistence of nontrivial solutions for all 
\bq\label{caso-delta2}
c\in \{0\}\cup (\sqrt{2a},\infty).
\eq
In particular, considering $a=1$, we recover Theorem \ref{teo:brezis}
in the cases \ref{caso-c-0} and \ref{teo:gravejat}.

So far, in view of \ref{H-reg}, we have assumed that $\wh W$ is regular in a neighborhood
of the origin,
which in particular allows us to define $c_s(W)$. 
However there are interesting examples of kernels provided by the physical literature 
 such that $\wh W$ is not continuous at the origin  and then $c_s(W)$ is not properly
defined. For this reason we will work with
a more general geometric condition on $\wh W$. 
More precisely, 
 denoting by $\{e_k\}_{k\in \{1,\dots ,N\}}$  the canonical unitary vectors of $\R^N$, we introduce the function  
 \bq\label{def-w-chico}
\w_j(\nu_1,\nu_2):= \wh W(\nu_1 e_1+\nu_2 e_j),\quad (\nu_1,\nu_2)\in \R^2, \ j\in \{2,\dots,N\},
\eq
and the set
\bqq
\Gamma_{j,c}:=\{ \nu=(\nu_1, \nu_2) \in \R^2 : \abs{\nu}^4+2\w_j(\nu)\abs{\nu}^2-c^2\nu_1^2=0  \}.
\eqq
Then Theorem \ref{teo1} can be generalized if we replace  \ref{H-reg} by the condition
 \begin{enumerate}[resume*]
\item\label{H-gamma} 
For all $j\in\{2,\dots, N\}$ and $c>0$, there exist $\delta>0$ and  two functions $\gamma^+_{j,c}$ and $\gamma^-_{j,c}$, 
defined on the interval $(0,\delta)$, such that
the set $\Gamma_{j,c}\cap B(0,\delta)$ has Lebesgue measure zero, 
 $\gamma_{j,c}^\pm\in C^1((0,\delta))$, and
$$\gamma^+_{j,c}(t)>0, \ \ \gamma^-_{j,c}(t)<0,\ \ (t,\gamma_{j,c}^\pm(t))\in \Gamma_{j,c}, \quad \text{for all } t\in (0,\delta).$$
Moreover, the following limits exist and are equal 
\bqq
\lim_{t\to 0^+}\left(\frac{\gamma_{j,c}^+(t)}{t}\right)^2= \lim_{t\to 0^+}\left(\frac{\gamma_{j,c}^-(t)}{t}\right)^2=:{\ell_{j,c}}.
\eqq
\end{enumerate}
\begin{figure}[ht]
\begin{center}
\scalebox{0.9}{\includegraphics{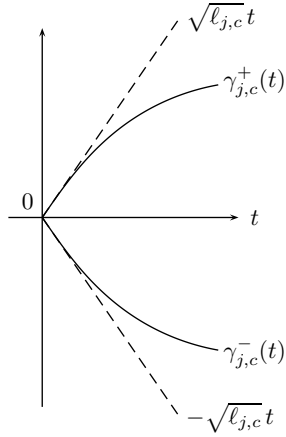}}
\end{center}
\caption{The curves $\gamma^{\pm}_{j,c}$ of condition \ref{H-gamma}.}
\label{figura}
\end{figure}
Figure~\ref{figura} illustrates  condition \ref{H-gamma}. The fact that  \ref{H-reg} and \eqref{lim-W}  actually imply \ref{H-gamma}
  is proved in Section~\ref{sec-morse}
(see Lemma~\ref{curva}). We also note that from \ref{H-gamma} we infer that   $\lim_{t\to 0^+}\gamma^\pm_{j,c}(t)=0$. Moreover, 
if $\wh W$ is even in each component, that is 
\bqq
\wh W((-1)^{m_1}x_1,(-1)^{m_2}x_2,\dots,
(-1)^{m_N} x_N)=\wh W(x_1,x_2,\dots,x_N),
\eqq
for all $(m_1,\dots,m_N)\in \{0,1\}^N,$
then $\gamma_{j,c}^-=-\gamma_{j,c}^+$, for all $j\in \{2,\dots,N\}$.

On the other hand, if the values $\ell_{j,c}$ are positive,
a necessary condition for the existence of a nontrivial finite energy solution of \eqref{TW}
is that they  are equal.
\begin{lema}\label{iguales}
Let $c>0$. Assume that $W$ satisfies  \ref{H0}--\ref{W-positivo} and \ref{H-gamma} with \mbox{$\ell_{j,c}>0$,} for all $j\in \{2,\dots, N\}$.
Let $v\in \E(\R^N)$ be a nontrivial solution of \eqref{TW} in $\E(\R^N)$. 
Then 
\bqq
\ell_{1,c}=\ell_{2,c}=\dots=\ell_{N,c}.
\eqq
\end{lema}
Now we are ready to state our main result in its general form.
\begin{teo}\label{teorema}
Let $c>0$. Assume that  $W$ satisfies  \ref{H0}--\ref{W-positivo} and \ref{H-gamma},
with 
  \bq\label{def-l-c}
\ell_c:=\ell_{1,c}=\ell_{2,c}=\dots=\ell_{N,c}>0.
\eq
 Suppose  that there exist constants  $\sigma_1,\dots,\sigma_N\in \R$  such that
\bq
\label{sigma-1} \wh W(\xi)+\ell_c \sum_{k=2}^N\sigma_k \xi_k\ptl_k \wh W(\xi)-\sigma_1 \xi_1\ptl_1\wh W(\xi)\geq  0, \text{ for a.a. } \xi \in \R^N, \eq
and
\bq
\label{sigma-2} \sum_{k=2}^N \sigma_k +\min\left\{-\sigma_1-1,\frac{\sigma_1-1}{\ell_c+2},2\ell_c \sigma_j+\sigma_1-1\right\}\geq 0,  
\eq
 for all  $j\in \{2,\dots, N\}$. Then nontrivial solutions of \eqref{TW} in $\E(\R^N)$ do not exist.  
\end{teo}
Finally, we give the corresponding analogue of Corollaries \ref{cor-teo1}--\ref{cor-teo}.
\begin{cor}\label{ultimo-cor}
 Let $c>0$. Assume that  $W$ satisfies  \ref{H0}--\ref{W-positivo}, \ref{H-gamma} and \eqref{def-l-c}.
Suppose that either \eqref{hyp-cor}  or  
$$l_c\leq \inf_{\xi\in \R^N} \frac{(N-1)\wh W(\xi)}{\sum_{k=2}^N \abs{\xi_k\ptl_k \wh W(\xi)}}$$
hold. Then nontrivial solutions of \eqref{TW} in $\E(\R^N)$ do not exist. 
\end{cor}
\subsection{Examples}\label{examples}
In this subsection we provide some potentials of physical interest for which 
the Cauchy problem for
\eqref{NGP}  is globally well-posed (see  \cite{delaire}). 

(I) Given the spherically symmetric
interaction of particles, in physical models it is usual to suppose that $W$ is radial
and then so is its Fourier transform, namely
$$\widehat W(\xi)=\rho(\abs{\xi}),$$
for some function $\rho:[0,\infty)\to \R$. Assuming that  $\rho$ is differentiable, we compute
\bq\label{laderivada}
\xi_k \ptl_k \wh W(\xi)=\rho'(\abs{\xi})\frac{\xi_k^2}{\abs{\xi}}, \quad \text{ for all }\xi\in \R^N\setminus\{0\}.
 \eq
Then, using that $\sum_{k=2}^N \xi_k^2=\abs{\xi}^2-\xi_1^2$ and that $\abs{\xi_k}\leq \abs{\xi}$, we obtain that 
 conditions \eqref{hyp-cor} and \eqref{ine:cor} are respectively satisfied if
\bq
 \max\left\{1,\frac{2}{N-1}\right\} \leq \inf_{r>0}  \frac{\rho(r)}{\abs{\rho'(r)}r},\label{ex:1}\\
\eq
and
\bq
 2\rho(0)< c^2\leq 2\rho(0)     \left(   1+   \inf_{r>0} \frac{\rho(r)}{\abs{\rho'(r)}r} \right).\label{ex:2}
\eq
We consider now a generalization of the 
model proposed by Shchesnovich and Kraenkel \cite{kraenkel}
\bqq
\rho(r)=\frac{1}{(1+a r^2)^{b/2}}, \quad a,b>0,
\eqq
so that 
$$c_s:= c_s(W)=\sqrt2.$$
It is immediate to verify that hypotheses \ref{H0},\ref{H-resto}--\ref{H-reg} are satisfied. 
Also, since $\wh W\in L^\infty(\R^N)$, \ref{H-infty} is fulfilled for $N=2,3$. Moreover, 
by Proposition 6.1.5 in \cite{grafakos}, we conclude that 
$W\in L^1(\R^N)\cap L^N(\R^N)$ for $N\geq 4$ provided that 
$b>N-1$. On the other hand,
\bq\label{ex:3}
\inf_{r>0}  \frac{\rho(r)}{\abs{\rho'(r)}r}=\inf_{r>0}\frac{1+ar^2}{abr^2}=\frac{1}{b}.
\eq
Therefore, using \eqref{laderivada}--\eqref{ex:3} and invoking  Corollaries \ref{cor-teo1},\ref{cor-teo} and Theorem \ref{teo-c-0},
we conclude that in the following cases there is nonexistence of nontrivial solutions of \eqref{TW} in $\E(\R^N)$
\begin{enumerate}[label=(\alph*),ref=(\alph*)]
\item\label{ex-a} $N=2$, $b\leq 1/2$, $c\in (c_s,\infty)$.
\item $N=2$,  $b> 1/2$, $c\in(c_s,\sqrt{2+2/b})$.
\item\label{ex-b} $N=3$, $b\leq 1$, $c\in (c_s,\infty)$.
\item $N=3$,  $b> 1$, $c\in (c_s,\sqrt{2+2/b})$.
\item $N\geq 4$,  $b> N-1$, $c\in (c_s,\sqrt{2+2/b})$.
\item $N=2$ or 3,  $c=0$.
\item $N\geq 4$,  $b> N-1$, $c=0$.
\end{enumerate}
We remark that if $b\to 0$, $\wh W\to 1$ and then $W\to \delta$
in a distributional sense. Thus the cases \ref{ex-a} and \ref{ex-b} could be seen  as a generalization
of Theorem~\ref{teo:brezis} in the cases \ref{caso-c-0} and \ref{teo:gravejat}.

(II) Let $N=2,3$  and 
\bqq
W_\ve=\delta +\ve f,\quad \ve\geq 0,\eqq
where $f$ is an even real-valued function, 
such that $f, \abs{x}^2f, \abs{x}\grad{f} \in L^1(\R^N)$.
 Then $\wh W_\ve=1+\ve \wh f\in C^2(\R^N)$. Since  
\bq\label{identidad-fourier}
\wh{x_j \ptl_k f}=-(\delta_{j,k}\wh f+\xi_k\ptl_j \wh f),
\eq
we have
\bqq
\norm{\wh f}_{L^\infty(\R^N)}\leq \norm{f}_{L^1(\R^N)},\quad \norm{\xi_k\ptl_j \wh f}_{L^\infty(\R^N)}\leq \norm{f}_{L^1(\R^N)}+ 
\norm{x_j\ptl_k{f}}_{L^1(\R^N)}.
\eqq
Then we see that $W$ satisfies conditions 
\ref{H0}--\ref{H-reg} provided that $\ve <\norm{f}_{L^1(\R^N)}^{-1}$ and that the sonic speed given by
$$c_s:= c_s(W)=\left(2+2\ve\intR f\right)^{1/2},$$
is well-defined.
Moreover  \eqref{hyp-cor} is fulfilled if
\bq\label{condition-ve}
\ve<\left(4\norm{f}_{L^1(\R^N)}+\sum_{k=1}^N\norm{x_k \ptl_k f}_{L^1(\R^N)}  \right)^{-1}.
\eq
Therefore, under condition \eqref{condition-ve}, Corollary \ref{cor-teo1} implies the 
nonexistence of nontrivial solutions of \eqref{TW} in $\E(\R^N)$ for any $c\in (c_s,\infty).$

(III) The following potential used in \cite{remi,yi} to  model dipolar forces in a quantum
gas 
yields an example in $\R^3$ where the speed of sound is not properly defined. Let 
\bqq
W=a\delta+b K,\qquad a,b\in \R, 
\eqq
 where $K$ is the singular kernel
\bqq K(x)=\dfrac{x_1^2+x_2^2-2x_3^2}{\abs{x}^5}, \quad x\in\R^3\backslash\{0\}.\eqq
In the sequel, we will deduce from  Lemma~\ref{iguales} and Theorem~\ref{teorema} 
that there is nonexistence
of nontrivial finite energy solutions of \eqref{TW} in $\E(\R^N)$ for all 
\bq\label{condicion-c}
(2\max\{ {a-\tilde b,a}\})^{1/2}<c<\infty,
\eq
with  $\tilde b=(4\pi b)/3$, provided that $a>0$ and either
\bq\label{alpha_}
a\geq \tilde b\geq  0 \qquad\textrm{or}\qquad a> -2\tilde b\geq  0.\eq

We now turn to the proof of condition \eqref{condicion-c}. In fact, since (see \cite{remi})
 $$\wh{W}(\xi)=a+\tilde b\left(\frac{3\xi_3^2}{\abs{\xi}^2}-1\right), \qquad \xi\in\R^3\backslash\{0\},$$
$W$ satisfies \ref{H0}--\ref{W-positivo} if one of the conditions in \eqref{alpha_} holds.
However, $\wh W$ is not continuous at the origin. More precisely, in terms of the function defined in \eqref{def-w-chico}, 
we have that $\w_2$ is constant equal to $a>0$ and
by Lemma~\ref{curva} there exist curves $\gamma_2^{\pm}$ with
$\ell_{2,c}={c^2}/(2a)-1$. On the other hand, $\w_3$ is not continuous at the origin but 
assuming \eqref{condicion-c} we can explicitly 
solve the algebraic 
equation
$$(x^2+y^2)^2+2\w_3(x,y)(x^2+y^2)-c^2x^2=0$$
and deduce that
\bqq
 \gamma^\pm_{3,c}(t)=
\pm\sqrt{-t^2-a-2\tilde b+\sqrt{6\tilde bt^2+ (a+2\tilde b)^2+c^2t^2}},
\eqq
for $\abs{t}<c^2-2(a-\tilde b)$. 
Therefore \ref{H-gamma} holds and  $\ell_{3,c}=-1+(6\tilde b+c^2)/(2(a+2\tilde b))$.
Note that by  \eqref{alpha_}, $\ell_{3,c}$ is a well-defined positive constant.
By Lemma~\ref{iguales},  a necessary condition
so that the equation \eqref{TW} has nontrivial solutions is  
$\ell_{3,c}=\ell_{2,c}$, which leads us to 
\bqq
(c^2-3a)b=0.
\eqq
The case  $b=0$ has already been  analyzed (see \eqref{caso-delta2}).
If  $b\neq 0$, we obtain  $c^2=3a$. Hence
$\ell_c:=\ell_{2,c}=\ell_{3,c}=1/2.$ Then, taking $\sigma_1=0$ and $\sigma_2=\sigma_3=1/2$,
\eqref{sigma-2} is satisfied and the l.h.s. of \eqref{sigma-1} reads
\bqq
a+\tilde b\left( 3\frac{\xi^2_3}{\abs{\xi}^2} \left(1- \frac{\xi^2_2}{2\abs{\xi}^2} \right)-1\right)+
\frac{3\tilde b}{2}
\frac{\xi^2_3}{\abs{\xi}^2}\left(1- \frac{\xi^2_3}{\abs{\xi^2}}  \right), 
\eqq
which is nonnegative by \eqref{alpha_}.
 Therefore, by Theorem~\ref{teorema}, 
there is nonexistence
of nontrivial solutions of \eqref{TW} in $\E(\R^N)$,
provided that \eqref{condicion-c} and \eqref{alpha_} hold.  

As proved in \cite{delaire}, the Cauchy problem is also globally well-posed
for other interactions such as the soft core potential 
\begin{equation*}
W(x)=\begin{cases}
1, & \ \text{ if }\abs{x}<a,\\
0, & \ \text{otherwise},
\end{cases}
\end{equation*}
with $a>0$. However, our results do not apply to this kernel,
since the changes of sign of $\wh W$  will prevent that an inequality such as
\eqref{sigma-1} can be satisfied. Moreover, in this case the energy could be negative
making more difficult the analysis. Nevertheless, $\wh W$ is positive near the origin
and the sonic speed is still well defined, 
so that it is an open question to establish which are the exact implications of
 change of sign of the Fourier transform in the nonexistence results. 

\subsection{Outline of the proofs and organization of the paper }
We recall that Theorem~\ref{teo:brezis}-\ref{caso-c-0} follows 
from a classical Pohozaev identity.
Gravejat in \cite{gravejat} proves Theorem~\ref{teo:brezis}-\ref{teo:gravejat}
by combining the respective Pohozaev identity with an integral equality
obtained from the Fourier analysis of the equation satisfied by $1-\abs{v}^2$.
Our results are derived in the same spirit. In the next section  we prove
that conditions \ref{H0} and \ref{H-infty} imply the regularity of solutions of \eqref{TW}. 
In Section~\ref{fourier} we prove that condition \ref{H-gamma} allows us to generalize
the arguments in \cite{gravejat} so that we can 
derive the integral identity \eqref{id-fourier}. The fact that the set $\Gamma_{j,c}$
is  described by the curves $\gamma_{j,c}^\pm$ is a consequence of the Morse lemma,
as explained in Section~\ref{sec-morse}.

In Section~\ref{sec:pohozaev} we establish a Pohozaev identity for \eqref{TW}
 with a ``remainder term'' depending on the derivatives
of $\wh W$. Although this identity can be formally obtained for rapidly decaying functions,
its proof for  functions in $\E(\R^N)$ is the major technical difficulty of this paper
and relies on Fourier analysis and the fact that $W$ is even. As in \cite{brezis}, 
we then see in Section~\ref{main} that Theorem~\ref{teo-c-0} is 
as straightforward consequence of this relation.

In Section~\ref{main} we also show that we can recast the identities described above 
as a suitable  linear system  of equations for which we can invoke the Farkas lemma 
to obtain the nonexistence conditions given in Theorems~\ref{teorema} and \ref{teo1}.  
The corollaries stated in Subsection~\ref{statement} then follow by choosing the 
values of $\sigma_1,\dots,\sigma_N$ appropriately.

\paragraph{\bf Notations.}
 We adopt the standard notation $C(\cdot,\cdot, \dots )$ to represent a generic constant
that depends only on each of its arguments.
For any $x,y\in \R^N$, $z,w\in \C$, we denote 
the inner products in $\R^N$ and $\C$, respectively, by $x.y=\sum_{i=1}^N x_iy_i$ and
$\langle z,w\rangle=\Re(z\overline w)$.  The Kronecker delta  $\delta_{k,j}$
takes the value one if $k=j$ and zero otherwise.  $\mathcal F(f)$ or $\wh f$ stand for the  Fourier transform of $f$, namely 
$$\mathcal F(f)(\xi)=\wh f(\xi)=\intR f(x)e^{-ix.\xi}\,dx,$$
and $\F^{-1}$ for its inverse. 
 
From now on we fix $c\geq 0$. We denote by $v=v_1+i v_2$ ($v_1$, $v_2$ real-valued) a solution of \eqref{TW} in $\E(\R^N)$. We also set
the real-valued functions
 \bqq
\rho:=\abs{v}=(v_1^2+v_2^2)^{1/2},\quad \eta:=1-\abs{v}^2.
\eqq

\section{Regularity of solutions}\label{section-regularity}
\begin{lema}
\label{regularity} 
 Assume that $W\in \M_{2,2}(\R^N)$. Then  $v\in W^{2,\,4/3}_{\loc}(\R^N)$.
Suppose further that  $2\leq N\leq 3$. Then $v$ is smooth and bounded. Moreover,  $\eta$ and $\grad{v}$ belong to $W^{k,p}(\R^N)$,
for all $k\in\N$, $2\leq p\leq \infty$.
\end{lema}
\begin{proof}
Let $\bar x\in \R^N$ and $B_r:= B(\bar x, r)$ the ball of center $\bar x$ and radius $r$.
Then
\bq\label{L4}
\norm{v}_{L^4(B_1)}=\norm{\abs{v}^2}_{L^2(B_1)}\leq \norm{\abs{v}^2-1}_{L^2(\R^N)}+\norm{1}_{L^2(B_1)}\leq E(v)+C(N).
\eq
On the other hand, we can decompose $v$ as $v=z_1+z_2+z_3$, where $z_1,z_2$ and $z_3$ are  the solutions of the following equations
\begin{ecu}\label{descom1}
 -\Delta z_1&=0, &&\textrm{ in }  B_1,\\
z_1&=v, &&\textrm{ on } \ptl B_1,
\end{ecu}\label{descom2}
\begin{ecu}\label{tilde-q}
 -\Delta z_2&=i c  \ptl_1 v  , &&\textrm{ in }  B_1,\\
z_2&=0, &&\textrm{ on } \ptl B_1,
\end{ecu}
\begin{ecu}\label{descom3}
 -\Delta z_3&=v(W*\eta), &&\textrm{ in } B_1,\\
z_3&=0, &&\textrm{ on } \ptl B_1.
\end{ecu} 
Since $z_1$ is a harmonic function, 
$$\norm{z_1}_{C^k(B_{1/2})}\leq C(N,k,E(v)),$$
for all $k\in \N$. Using the H\"older inequality, \eqref{L4} and elliptic regularity estimates, we also have
\bqq
\norm{z_2}_{W^{2,2}(B_1)}\leq C(N,E(v)), \quad \norm{z_3}_{W^{2,4/3}(B_1)}\leq C(N,E(v))\norm{\wh W}_{L^{\infty}(\R^N)}\norm{\eta}_{L^2(\R^N)}.
\eqq
Therefore $\norm{v}_{W^{2,4/3}(B_{1/2})}\leq C(N,E(v),\eta, W)$. Furthermore, by the Sobolev embedding theorem we deduce 
that  $\norm{v}_{L^\infty(B_{1/2})}$ is bounded for $N=2$ and then this bound holds uniformly in $\R^2$. 
If $N=3$, we conclude that  $\norm{v}_{L^{12}(B_{1/2})}$ is uniformly bounded. Then
using the same decomposition \eqref{descom1}--\eqref{descom3} in the ball $B_{1/4}$, identical
arguments prove  that 
$\norm{v}_{W^{2,12/7}(B_{1/4})}\leq C(N,E(v),\eta, W)$, which by the Sobolev embedding theorem in dimension three implies that $\norm{v}_{L^\infty(B_{1/4})}$
is uniformly bounded. Consequently, $v\in L^\infty(\R^N)$ for $N=2,3$. 

Finally, using again  \eqref{descom1}--\eqref{descom3} and a
standard bootstrap argument, we conclude that $v\in W^{k,\infty}(\R^N)$ for all $k\in \N$.

Now, setting $w=\ptl_jv$, $j\in \{1,\dots, N\}$, and differentiating   \eqref{TW} with respect to $x_j$, we obtain for any $\lambda \in \R$
$$L_{\lambda}(w):=-\Delta w-ic \ptl_1 w+\lambda w=\ptl_jv (W*\eta)+v(W*\ptl_j \eta)+\lambda w,\quad \text{ in }\R^N.$$
Since $\grad v\in L^\infty(\R^N)\cap L^2(\R^N)$, we deduce that the r.h.s. belongs to $L^2(\R^N)$.
Then, for $\lambda>0$ large enough, we can apply the Lax--Milgram theorem to the operator $L_{\lambda}$ to deduce that $w\in H^{2}(\R^N)$.
Thus $\grad v\in H^2(\R^N)$ and a bootstrap argument shows that $\grad v\in H^{k}(\R^N)$, for all $k\in \N$
and therefore, by interpolation, $\grad v,\eta\in W^{k,p}(\R^N)$, for all $p\geq 2$ and $k\in \N$.
\end{proof}

In Lemma \ref{regularity}, we needed to differentiate the equation \eqref{TW} to improve the regularity,
which required that $W*\grad \eta$ was well-defined. If $N\geq 4$, 
proceeding as in Lemma \ref{regularity}, we can only infer that $\grad \eta \in L^{4/3}_{\loc}(\R^N)$
so that it is not clear that we can give a sense to the term $W*\grad \eta$.
On the other hand, if $N\geq 3$, the fact that $\grad v \in L^2(\R^N)$ implies that there exists $z_0\in \C$ with $\abs{z_0}=1$
such that $v-z_0\in L^{\frac{2N}{N-2}}(\R^N)$ (see e.g. \cite[Theorem 4.5.9]{hormander}).
Moreover, since \eqref{TW} is invariant by a change of phase, 
we can assume that
$v-1\in L^{\frac{2N}{N-2}}(\R^N)$. Therefore, 
\bq\label{est-eta}
\grad \eta=-2\langle v-1,\grad v\rangle-2\langle 1,\grad v\rangle \in L^{N/(N-1)}(\R^N)+L^2(\R^N).
\eq
Then it would be reasonable to suppose that $W\in \M_{N/{N-1},q}(\R^N)$, for some $q\geq {N}/{N-1}$.
However, this is not enough to invoke the elliptic regularity estimates
and that is reason why we work with the assumption \eqref{hyp-high} in \ref{H-infty} if $N\geq 4$.
We remark that to establish precise conditions on $W$  that ensure the regularity of solutions
of \eqref{TW} in higher dimensions goes beyond the scope of this paper.

\begin{lema}\label{reg-high}
Let $N\geq 4$. Assume that $W$ satisfies \ref{H-infty}.
 Then $v$ is bounded and smooth.
Moreover,  $\eta$ and $\grad{v}$ belong to $W^{k,p}(\R^N)$, for all $k\in\N$, $2\leq p\leq \infty$.
\end{lema}
\begin{proof}
From \eqref{hyp-high}, by duality (see e.g. \cite{grafakos}) we infer that  
 $W\in \M_{1,N}(\R^N)\cap \M_{1,2N/(N+2)}(\R^N)$. Then, from the Riesz--Thorin interpolation theorem
and the fact that 
$(1/2,(N-2)/(2N))$ and $((N-1)/N,(N-2)/(2N))$ belong
to the convex hull of
 $$\left\{\left(\frac{1}{2},\frac12 \right),
\left(\frac{N-1}{N},0\right),\left(\frac{N-2}{2N},0\right),\left(1,\frac1N\right),\left(1,\frac{N+2}{2N}\right)\right\},$$
we conclude that 
\bq\label{M1}
W\in \M_{2,2N/(N-2)}(\R^N)\quad {\rm and}\quad W\in \M_{N/(N-1),2N/(N-2)}(\R^N).
\eq
As mentioned before, we can assume that  $\tilde v:=v-1\in L^{\frac{2N}{N-2}}(\R^N)$. Then using \ref{H-infty},
 \eqref{est-eta} and \eqref{M1}, we are led to 
\bq\label{dem-high} W*\eta,W*\grad \eta \in L^\infty(\R^N)\cap L^{2N/(N-2)}(\R^N).
\eq
Now we recast  \eqref{TW} as
\bq\label{eq-high}
 L_{\lambda}(\tilde v):=-\Delta \tilde v-ic \ptl_1 \tilde v+\lambda \tilde v=\tilde v((W*\eta)+\lambda)+W*\eta,\quad \text{ in }\R^N,
\eq
for some $\lambda>0$. By \eqref{dem-high}, the r.h.s. of \eqref{eq-high} belongs to $L^{2N/(N-2)}(\R^N)$. Then choosing
$\lambda$ large enough, we can apply elliptic regularity estimates to the operator $L_\lambda$ to conclude that $\tilde v\in W^{2,2N/(N-2)}(\R^N)$.
Then 
\bqq
\ptl_{j,k}\eta=-2(\langle v-1,\ptl_{j,k} v\rangle +\langle \ptl_jv,\ptl_kv\rangle + \langle 1,\ptl_{j,k} v\rangle)\in L^{N/(N-1)}(\R^N)+L^{2N/(N-2)}(\R^N),
\eqq
for any $1\leq j,k\leq N$. Therefore, by \eqref{hyp-high} and \eqref{M1}, $W*\ptl_{j,k}\eta \in L^\infty(\R^N)\cap L^{2N/(N-2)}$. Thus
the r.h.s. of \eqref{eq-high} belongs to $W^{2,2N/(N-2)}(\R^N)$, so that 
$\tilde v\in W^{4,2N/(N-2)}(\R^N)$. A bootstrap argument yields that  $\tilde v\in W^{k, 2N/(N-2)}(\R^N)$,
for any $k\in \N$. By the Sobolev embedding theorem, we conclude that $v\in W^{k,\infty}(\R^N)$ for any $k\in \N$.
Then the conclusion follows as in Lemma~\ref{regularity}.
\end{proof}
\begin{lema}\label{lema-reg2}
 Let $W\in L^1(\R^N)$ if $2\leq N\leq 3$ and $W\in L^1(\R^N)\cap L^N(\R^N)$ if $N\geq 4$.
Then $W$ fulfills \ref{H-infty}.
\end{lema}
\begin{proof}
Since $W\in L^1(\R^N)$, 
by the Young inequality we have 
\bqq
\norm{W*f}_{L^p(\R^N)}\leq \norm{W}_{L^1(\R^N)}\norm{f}_{L^p(\R^N)},\quad \text{ for any } p\in [1,\infty].
\eqq
Then, taking $p=2$, we conclude that \ref{H-infty} holds for $2\leq N\leq 3$.
For $N\geq 4$, we have
 $W\in L^1(\R^N)\cap L^N(\R^N)$. In particular, ${W}\in {L^{2N/(N+2)}(\R^N)}$ and the Young inequality
implies that
\begin{align*}
 \norm{W*f}_{L^\infty(\R^N)}&\leq \norm{W}_{L^N(\R^N)}\norm{f}_{L^{N/(N-1)}(\R^N)},\\
\norm{W*f}_{L^\infty(\R^N)}&\leq \norm{W}_{L^{2N/(N+2)}(\R^N)}\norm{f}_{L^{2N/(N-2)}(\R^N)}.
\end{align*}
Therefore   \ref{H-infty} is satisfied.
\end{proof}
 \begin{cor}\label{cor-reg}
 Assume that $W$ satisfies \ref{H-infty}.  
Then $v$ is smooth and bounded.
Moreover,  $\eta$ and $\grad{v}$ belong to $W^{k,p}(\R^N)$, for all $k\in\N$, $2\leq p\leq \infty$, and
\bq\label{cor-conv}
{\rho(x)}\to 1, \ \grad v(x)\to 0, \qquad \text{as} \ \abs{x}\to\infty.
\eq
Furthermore, there exists a smooth lifting  of $v$. More precisely, there exist $R_0>0$
and a smooth real-valued function $\theta$ defined on $B(0,R_0)^c$, with $\grad \theta\in W^{k,p}(B(0,R_0)^c)$,
for all $k\in\N$, $2\leq p\leq \infty$,  such that
\bq\label{lifting}
\rho\geq \frac12\quad \text{ and }\quad v=\rho e^{i\theta}\quad \text{ on } B(0,R_0)^c.
\eq
 \end{cor}
\begin{proof}
The first part is exactly Lemmas \ref{regularity} and \ref{lema-reg2}. In particular, 
$v$ and $\grad v$ are  uniformly continuous on $\R^N$.
Then, since $1-\abs{v}^2\in L^2(\R^N)$ and ${\grad v}\in L^2(\R^N)$, we obtain \eqref{cor-conv}.
The existence of the lifting satisfying \eqref{lifting} follows as in \cite[Proposition 2.5]{maris-non}. From  
\eqref{lifting} we also deduce that
\bqq
\abs{\grad v}^2=\abs{\grad \rho }^2+\rho^2\abs{\grad \theta}^2 \quad \text{ on } B(0,R_0)^c.
\eqq
Since $\rho\geq 1/2$ on $B(0,R_0)^c$, we infer that $\grad \theta\in W^{k,p}(B(0,R_0)^c)$,
for all $k\in\N$, $2\leq p\leq \infty$.
\end{proof}

 In virtue of Corollary \ref{cor-reg}, we introduce the function $\phi\in C^\infty(\R^N)$, 
 $\abs{\phi}\leq 1$, such that $\phi=0$ on $B(0,2R_0)$ and $\phi=1$ on $B(0,3R_0)^c$.
In this way,  we can assume the function $\phi \theta$ is well-defined on $\R^N$.
This will be useful in the next section to work with global functions in terms of $\theta$.
In fact, we end this section with the following result.

\begin{lema}\label{lemma-G}
 Assume that $W$ satisfies \ref{H-infty}.
Then \bq\label{G}
G:=v_1\grad v_2-v_2\grad v_1-\grad(\phi \theta),\quad \text{on } \R^N,\eq
 belongs to $ W^{k,p}(\R^N)$, for all $k\in \N$ and  $1\leq p\leq \infty$.
\end{lema}
\begin{proof}
 By Corollary \ref{cor-reg}, $G\in C^\infty(\R^N)$ and moreover 
\bqq
G=-\eta\grad \theta \quad \text {on }B(0,3R_0)^c.
\eqq
Since $\grad \theta\in W^{k,p}(B(0,R_0)^c)$ and $\eta \in W^{k,p}(B(0,R_0)^c)$, for all $k\in \N$, $2\leq p\leq \infty$,
the conclusion follows.
\end{proof}
\section{An integral identity}\label{fourier}
The aim of this section is to prove the following integral identity.
\begin{prop}\label{prop-fourier}
Let $c>0$. Suppose that \ref{H-infty} and \ref{H-gamma} hold with $\ell_{j,c}>0$, for some
$j\in \{2,\dots,N\}$. Then  
\bq\label{id-fourier}
\intR (\abs{\grad v}^2+\eta(W*\eta))=
-c\frac{\ell_{j,c}}{1+\ell_{j,c}}\intR{(v_1\ptl_1 v_2-v_2\ptl_1 v_1-\ptl_1(\phi \theta))}.
\eq
\end{prop}
We note  that  since  $W$ satisfies \ref{H-infty},  all the results
 of Section~\ref{section-regularity} hold. On the other hand, from \eqref{TW} we deduce that 
  $\eta=1-\abs{v}^2$ satisfies 
\bq\label{dem-eta1}
\Delta \eta=-F+2W*\eta-2c\ptl_1(\phi\theta).
\eq
where 
$$F:=2\abs{\grad v}^2+2\eta(W*\eta)+2cG_1.$$
and
$G=(G_1,\dots,G_N)$ was defined in \eqref{G}. Considering real and imaginary parts in \eqref{TW} and multiplying them by $v_2$ and $v_1$, respectively,
it follows that 
\bq\label{dem-eta2}
\div(G)=v_1\Delta v_2-v_2\Delta v_1-\Delta (\phi \theta)=\frac c2\ptl_1\eta-\Delta(\phi \theta).
\eq
Therefore, from \eqref{dem-eta1} and \eqref{dem-eta2}, we conclude that 
\bq\label{bi}
\Delta^2 \eta -2\Delta (W*\eta)+c^2\ptl_{11}^2\eta=-\Delta F+2c\ptl_1(\div G), \quad \text{ in }\R^N.
\eq
Since we are assuming \ref{H-infty}, by Corollary \ref{cor-reg} and Lemma~\ref{lemma-G}, 
we have that $F,G \in W^{k,1}(\R^N)\cap W^{k,2}(\R^N)$, for all $k\in \N$,
so that \eqref{bi} stands in $L^2(\R^N)$. Taking the Fourier transform in equation \eqref{bi}
and setting
\bqq
R(\xi):=\abs{\xi}^4+2\wh W(\xi)\abs{\xi}^2-c^2\xi_1^2\quad \text{and}\quad H(\xi):=\abs{\xi}^2 \wh F(\xi)-2c\sum_{j=1}^N \xi_1\xi_j \wh G_j(\xi),
\eqq
we get
\bq\label{id-int}
R(\xi)\wh \eta(\xi)=H(\xi),\quad \text{ in }L^2(\R^N).
\eq
\begin{lema}\label{H-cero}
Let $c>0$. Suppose that \ref{H-infty} and  \ref{H-gamma} hold. Then for all
$j\in \{2,\dots,N\}$, 
\bq\label{H-igual}
H(te_1+ \gamma^\pm_{j,c}(t) e_j)=0, \quad \text{ for all } t\in (0,\delta),
\eq
where $\delta$ is given by \ref{H-gamma}.
\end{lema}
\begin{proof}
We fix $j\in \{2,\dots,N\}$ and we prove \eqref{H-igual} for $\gamma^+_{j,c}$, since the proof for  $\gamma^-_{j,c}$ is  analogous. 
To simplify the notation, we put $\gamma:= \gamma^+_{j,c}$. As stated before,
  $F,G \in W^{k,1}(\R^N)\cap W^{k,2}(\R^N)$, for all $k\in \N$. In particular $F, G\in L^1(\R^N)$,
so that $\wh F$, $\wh G\in C(\R^N)$. Thus $H$ is a continuous function on $\R^N$. 

Let $\delta>0$ given by \ref{H-gamma}. Arguing by contradiction, we suppose that there exist $t_0\in (0,\delta)$ and a constant $A>0$
such that $\abs{H(\tilde \xi)}\geq A$, where $\tilde \xi=t_0e_1+ \gamma(t_0) e_j$. By the continuity of $H$,
there exists $r>0$ such that 
$\abs{H(\xi)}\geq A$, for all $\xi \in V_r$, where
 $$V_r=B(\tilde \xi,r)\cap \{\alpha e_1+\beta e_j : \alpha,\beta \in \R\}.$$
Thus $V_r$ is a two-dimensional set and since $t_0>0$, we can choose
$r$ small enough such that $0\notin V_r$.
Then \eqref{id-int} yields
\bq\label{H-0-1}
\abs{\wh \eta(\xi)}^2\geq \frac{A^2}{(R(\xi))^2}, \quad \text{for all }\xi \in V_r\setminus \Gamma_{j,c}.
\eq
We claim that
\bq\label{H-0-2}
I:=\int_{V_r\setminus  \Gamma_{j,c}}\frac{d\xi_1 d\xi_j}{(R(\xi))^2}=+\infty.
\eq
Since by hypothesis $\Gamma_{j,c}\cap B(0,\delta)$ has measure zero,
\eqref{H-0-1} and \eqref{H-0-2} contradict that $\wh \eta\in L^2(\R^N)$.

To prove \eqref{H-0-2}, since $V_r$ is a two-dimensional set, we identify it as a subset of $\R^2$ and so that
we write $e_2$ instead of $e_j$. Then, since $\Gamma_{j,c}\cap B(0,\delta)$ has measure zero, 
$$I=\int_{V_r}\frac{d\xi_1 d\xi_2}{(R(\xi))^2}.$$
To compute the integral we ``straighten out'' the curve $\gamma$. Namely, we introduce the change of variables 
\begin{align*}
\xi_1&=\nu_1=:\Phi_1(\nu_1,\nu_2),\\
\xi_2&=\nu_2+\gamma(\nu_1)=:\Phi_2(\nu_1,\nu_2).
\end{align*}
Since $\gamma$ is a $C^1$-function, so is $\Phi$. 
Moreover, there is  some set $U_r$ such that $V_r=\Phi(U_r)$ and 
 $\abs{\det(J\Phi(\nu))}=1$  for  all $\nu\in U_r$. 
Setting $F(\nu):=R(\Phi(\nu))$, $\nu \in U_r$, the change of variables theorem yields
\bq\label{int-impropia}
I=\int_{U_r}\frac{d\nu_1 d\nu_2}{(F(\nu))^2}.
\eq
Furthermore, since $F\in C^1(U_r)$ and $F(\nu_1,0)=0$ for all $(\nu_1,0)\in U_r$, the Taylor theorem
implies that for any $(\nu_1,\nu_2)\in U_r$, there is some $\bar  \nu\in U_r$ such that
\bq\label{taylor-F}
F(\nu_1,\nu_2)=F(\nu_1,0)+\frac{\ptl F}{\ptl \nu_2}(\bar \nu)\nu_2=\frac{\ptl F}{\ptl \nu_2}(\bar \nu)\nu_2, 
\eq
 On the other hand, 
by \ref{H-infty}, $\wh W\in L^\infty(\R^N)$ and by 
\ref{H-resto}, $\grad W\in L^\infty(V_r)$, so that 
$\norm{ \wh W}_{W^{1,\infty}(V_r)}<\infty$.
Thus $\norm{\grad F}_{L^{\infty}(U_r)}\leq C(r,\gamma)(1+\norm{ \wh W}_{W^{1,\infty}(V_r)})$
and from \eqref{taylor-F} we conclude that
\bq\label{cota-abajo}
\abs{F(\nu)}\leq C(r,\gamma)(1+\norm{ \wh W}_{W^{1,\infty}(V_r)}) \abs{\nu_2},  \quad \text{for all }\nu \in U_r.
\eq
 From \eqref{int-impropia} and \eqref{cota-abajo}, taking $\tilde \nu=(\tilde \nu_1,\tilde \nu_2)\in U_r$ such that $\tilde \xi=\Phi(\tilde  \nu)$
and  $\ve>0$ small enough, we conclude that 
$$I\geq C(r,\gamma,\wh W) \int_{U_r}  \frac{d\nu_1 d\nu_2}{\nu_2^2}\geq C(r,\gamma,\wh W)   
\int_{\tilde\nu_1-\ve}^{\tilde \nu_1+\ve} \int_{-\ve}^{\ve} \frac{d\nu_2 d \nu_1}{\nu_2^2}=+\infty,$$
which concludes the proof.
\end{proof}

Finally, we give the proof of identity \eqref{id-fourier}.
\begin{proof}[Proof of Proposition \ref{prop-fourier}]
By Lemma \ref{H-cero},  setting $\xi^{\pm}(t)=t e_1+\gamma_{j,c}^\pm (t)e_j,$ we have
\begin{align*}
(t^2+(\gamma_{j,c}^\pm(t))^2) \wh F(\xi^\pm(t))-2ct^2\wh G_1(\xi^\pm(t))-2ct \gamma_{j,c}^\pm(t) \wh G_j(\xi^\pm(t))=0,  \quad t\in (0,\delta).
\end{align*}
Dividing by $t^2$ and passing to the limit $t\to 0^+$,
\bqq
(1+\ell_{j,c})\wh F(0)-2c \wh G_1(0)  -2c\sqrt{\ell_{j,c}} \wh G_j(0)=(1+\ell_{j,c})\wh F(0)-2c \wh G_1(0) +2c\sqrt{\ell_{j,c}} \wh G_j(0)=0.
\eqq
Therefore, since $\ell_{j,c}> 0$, $\wh G_j(0)=0$ and $(1+\ell_{j,c})\wh F(0)=2c \wh G_1(0)$, which is precisely \eqref{id-fourier}.
\end{proof}

As a consequence of Proposition \ref{prop-fourier}, we obtain Lemma \ref{iguales}.
 \begin{proof}[Proof of Lemma \ref{iguales}]
 From \eqref{id-fourier}, setting
\bqq
 J(v)=\intR (\abs{\grad v}^2+\eta(W*\eta))\text{  and  }  P(v)=\intR{(v_1\ptl_1 v_2-v_2\ptl_1 v_1-\ptl_1(\phi \theta))},
\eqq
 we infer that 
 \bq\label{laultima}
\ell_{j,c}(J(v)+cP(v))=-J(v).
\eq
Since $v$ is nonconstant and $\wh W\geq 0$, we have that $J(v)>0$. 
Then we deduce from \eqref{laultima} that $J(v)+cP(v)\neq 0$ and 
$$\ell_{j,c}=-\frac{J(v)}{J(v)+cP(v)}.$$
Since the r.h.s. of the equality does not depend on $j$, the conclusion follows.
 \end{proof}
\section{The set $\Gamma_{j,c}$ under the condition \ref{H-reg}}\label{sec-morse}
In Section \ref{fourier} we have seen that identity \eqref{id-fourier} is a consequence
of the structure of the set $\Gamma_{j,c}$. More precisely, it relies on the fact that
\ref{H-gamma} provides the existence of $\delta>0$ and two curves $\gamma_{j,c}^\pm$ such that 
\bqq
\{(t,y^\pm(t)) : t\in (-\delta,\delta)\}\subseteq \Gamma_{j,c}.
\eqq
If $\wh W$ is of class $C^2$ in a neighborhood of the origin 
and $$\alpha_c:=\frac{c^2}{(c_s(W))^2}-1>0,$$
we can use the Morse lemma 
to justify the existence of the curves  $\gamma_{j,c}^\pm$ and to conclude that set $\Gamma_{j,c}$ 
consists of exactly these two curves near
the origin. Therefore the set $\Gamma_{j,c}$ looks like
Figure~\ref{curvas-morse} and condition \ref{H-gamma} is fulfilled.

\begin{figure}[ht]
\begin{center}
\scalebox{0.9}{\includegraphics{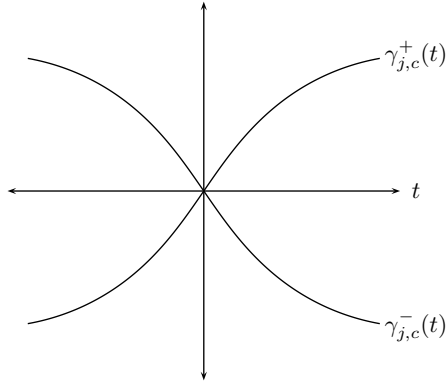}}
\end{center}
\caption{The set $\Gamma_{j,c}$ near the origin for $\wh W$ of class $C^2$.}
\label{curvas-morse}
\end{figure}

\begin{lema}\label{curva}
Assume that \ref{H0} and  \ref{H-reg} hold. Assume also that $\alpha_c>0$. Then, for each $j\in \{2,\dots,N\}$, 
there exist $\delta>0$ and  functions $y^\pm \in C^1((-\delta,\delta))\cap C^2((-\delta,\delta)\setminus\{0\})$
such that
\bq\label{eq-morse}
\Gamma_{j,c}\cap B(0,\delta)=\{(t,y^\pm(t)) : t\in (-\delta,\delta)\}.
\eq
Moreover, 
\bq\label{def-y}
\lim_{t\to 0^+} {y^\pm(t)}/{t}=\pm \sqrt{\alpha_{c}},
\eq
$y^+$ is strictly increasing 
and $y^-$ is strictly decreasing.
In particular,  \ref{H-gamma} is satisfied with $l_{j,c}=\alpha_{c}$.
\end{lema}
\begin{proof}
Let us set 
   \bqq
  R_j(\nu):=\abs{\nu}^4+2\w_j(\nu)\abs{\nu}^2-c^2\nu^2_1,\quad \nu=(\nu_1,\nu_2)\in \R^2.
  \eqq
In view of \ref{H-reg},  $R_j\in C^2(B(0,\delta_0))$, for some $\delta_0>0$.  Since 
$\w_j$ is even, we have that $\ptl_ 1\w_j(0,0)=\ptl_ 2\w_j(0,0)=0$. Then we obtain
$R_j(0,0)=0$, $\grad R_j(0,0)=0$,
\bq\label{dem-morse}
\frac{\ptl^2 R_j}{\ptl \nu_1^2}(0,0)=-4\alpha_c \w_j(0,0)<0, \ \ \frac{\ptl^2 R_j}{\ptl \nu_2^2}(0,0)=4\w_j(0,0)>0,\ \ \frac{\ptl^2 R_j}{\ptl \nu_1\ptl\nu_2}(0,0)=0.
\eq
 Therefore by the Morse lemma (see e.g. \cite[Theorem II]{ostrowski})
there exist two neighborhoods of the origin $ U,  V\subset \R^2$
and a local diffeomorphism $\Phi: U\to V$
such that 
\bq\label{eq-morse2}
R_j(\Phi^{-1}(z))= -2\alpha_c \w_j(0,0) z_1^2+2\w_j(0,0) z_2^2, \quad \text{ for all }z=(z_1,z_2)\in V.
\eq
Moreover, denoting $\Phi=(\Phi_1,\Phi_2)$ we have for $1\leq j,k\leq 2$
\bq\label{der-morse}
\frac{\ptl\Phi_j}{\ptl \nu_k }(\nu)\to \delta_{j,k}, \quad \text{ as }\abs{\nu}\to 0.
\eq
From \eqref{eq-morse2} we deduce that near the origin the set of solutions of $R_j=0$ 
is given  by the lines 
$$\{(t,\pm\sqrt{\alpha_c}t) : t\in (-\delta,\delta)\},$$
where we take $\delta>0$ such that the set is contained in $V$. Since $\Phi$ is a diffeomorphism we conclude that 
\bq\label{igual-morse}
\Gamma_{j,c}\cap B(0,\delta)=\{ (x_1^\pm(t),x_2^\pm(t)) : t\in (-\delta,\delta)\},
\eq
where
\begin{align}
\Phi_1(x_1^\pm ,x_2^\pm )&=t,\label{phi-1}\\
\Phi_2 (x_1^\pm ,x_2^\pm )&=\pm \sqrt{\alpha_c}t \label{phi-2}.
\end{align}
Moreover, differentiating relation \eqref{phi-1}
with respect to $t$ and using \eqref{der-morse}, we infer that $(x_1^\pm)'(t)\to 1$
as $t\to 0$. Therefore we can recast \eqref{igual-morse} as in \eqref{eq-morse}
with $y^\pm\in C^1((-\delta,\delta))\cap C^2((-\delta,\delta)\setminus\{0\})$.
Furthermore,  differentiating \eqref{phi-2} and using again \eqref{der-morse}
we conclude that 
$$(y^\pm)'(0)=\pm \sqrt{\alpha_c}.$$
Since $y^\pm\in C^1((-\delta,\delta))$,
taking a possible smaller value $\delta$, this implies \eqref{def-y}
and that  $y^+$ and  $y^-$ are  strictly increasing 
and decreasing on $(-\delta,\delta)$, respectively.
\end{proof}
\section{A Pohozaev identity}\label{sec:pohozaev}
In this section we  establish the following Pohozaev identity.
\begin{prop}\label{pohozaev}
Assume that \ref{H0}--\ref{H-resto} hold.  Then 
 \begin{align}
E(v)&=\intR{\abs{\ptl_1 v}^2}+\frac1{4(2\pi)^N}\intR  \xi_1 \ptl_1 \wh W \abs{\wh \eta}^2\,d\xi\label{poh1},\\
E(v)&=\intR{\abs{\ptl_j v}^2}-\frac{c}{2}\intR{(v_1\ptl_1 v_2-v_2\ptl_1 v_1-\ptl_1(\phi \theta))}+\frac1{4(2\pi)^N}\intR  \xi_j \ptl_j \wh W \abs{\wh \eta}^2\,d\xi,\label{poh2}
\end{align}
for all $j\in \{2,\dots,N\}.$
\end{prop}
Note that by Lemma~\ref{lemma-G}, $G_1=v_1\ptl_1 v_2-v_2\ptl_1 v_1-\ptl_1(\phi \theta)\in L^1(\R^N)$, thus every integral in \eqref{poh1} and \eqref{poh2}
is finite. As mentioned in Section~\ref{intro}, in the case that $W$ is the Dirac delta function
this result is well-known (see \cite{brezis,bethuel2,gravejat,maris-non}). The standard technique is to 
introduce a function $\chi\in C^\infty(\R)$, with $\chi(x)=1$ for $\abs{x}<1$ and $\chi(x)=0$ for $\abs{x}>2$,
and $\chi_n(x):=\chi(x/n)$. Then, multiplying \eqref{TW}  by $x_j\chi_n \ptl_j \bar v$ and taking real part, we are led to 
\begin{equation}\label{TW-1}
\langle ic\ptl_1 v +\Delta v,x_j\chi_n \ptl_j v\rangle-\frac{1}2 (W*\eta) x_j\chi_n \ptl_j \eta=0,  \quad\textrm{ on } \R^N,
\end{equation}
where we have used that
\bqq
\langle v,\ptl_j v \rangle=-\frac12 \ptl_j \eta.
\eqq
 Concerning  \eqref{TW-1}, we recall the following result.
\begin{lema}[\cite{brezis,bethuel2,gravejat,maris-non}]\label{lema:ultimo} Let $\varphi= \varphi_1+\varphi_2 \in \E(\R^N)\cap C^\infty(\R^N)$.
Assume that there exist $R^*>0$ and a smooth real-valued function $\tilde \theta$ defined on $B(0,R^*)^c$, with $\grad \tilde \theta\in L^2(B(0,R^*)^c)$,
such that
\bqq
\abs{\varphi}\geq \frac12\quad \text{ and }\quad \varphi=\abs{\varphi} e^{i\tilde \theta}\quad \text{ on } B(0,R_0)^c.
\eqq
Let  $\tilde \phi\in C^\infty(\R^N)$, such that $\tilde  \phi=0$ on $B(0,2R^*)$ and $\tilde  \phi=1$ on $B(0,3R^*)^c$.
Then for all $j\in \{1,\dots,N\}$, we have
 \begin{align}
&\lim_{n\to\infty} \intR \langle i\ptl_1 \varphi, x_j\chi_n \ptl_j \varphi\rangle=\frac 12(1-\delta_{1,j})  \intR{(\varphi_1\ptl_1 \varphi_2-
\varphi_2\ptl_1 \varphi_1-\ptl_1(\tilde  \phi \tilde \theta))},
\label{eq-po-2}\\
&  \lim_{n\to\infty} \intR \langle \Delta \varphi, x_j\chi_n \ptl_j \varphi \rangle=-\intR \abs{\ptl_j \varphi}^2+\frac12 \intR {\abs{\grad \varphi}}^2\label{eq-po-1},\\
&  \lim_{n\to\infty}-\frac{1}2\intR  x_j\chi_n (1-\abs{\varphi}^2) \ptl_j (1-\abs{\varphi}^2)=\frac{1}4\intR ({1-\abs{\varphi}^2})^2\label{caso-delta}.
 \end{align}
\end{lema}
Therefore, from  \eqref{TW-1} and Lemma~\ref{lema:ultimo}, Proposition \ref{pohozaev} follows in the case $W=\delta$.
To motivate our approach, let us briefly recall the proof of 
 \eqref{caso-delta}. First, we integrate by parts to obtain
\bqq
\begin{split}
A_n:&= -\frac{1}2\intR  x_j\chi_n (1-\abs{\varphi}^2) \ptl_j (1-\abs{\varphi}^2)\\
&=\frac12 \intR \chi_n(1-\abs{\varphi}^2)^2+
\frac12 \intR {x_j}  \ptl_j \chi_n(1-\abs{\varphi}^2)\ptl_j(1-\abs{\varphi}^2)-A_n.
\end{split}
\eqq
Then, invoking the dominated convergence theorem,
\bqq
A_n=\frac14 \intR \chi_n(1-\abs{\varphi}^2)^2+\frac14 \intR {x_j}  \ptl_j \chi_n(1-\abs{\varphi}^2) \ptl_j(1-\abs{\varphi}^2)\to 
\frac14 \intR (1-\abs{\varphi}^2)^2,
\eqq
as $n\to\infty$. In particular, we see 
that due to a symmetry property, we 
can write $A_n$ in  terms
of integrals to which we can apply 
the dominated convergence theorem.
However, in our nonlocal case we cannot use this trick and we
have to analyze the  integral associated to the potential energy more carefully. 
We rely in particular on the following general result.
 
\begin{prop}\label{lema-tecnico}
 Let  $f\in L^2(\R^N)\cap H_{\loc}^1(\R^N)$ be 
a real-valued function and $W\in \M_{2,2}(\R^N)$.
Assume also that \ref{H0}  and \ref{H-resto} hold. Then, for all $j\in\{1,\dots,N\}$,
\bq\label{limite-tecnico}
  \lim_{n\to\infty}-\frac12 \intR  (W*f) x_j\chi_n \ptl_j f=
 \frac{1}{4} \int_{\R^N} ( W * f ) f - \frac{1}{4(2\pi)^N} \int_{\R^N} {\xi_j \partial_j \wh W}(\xi) \abs{\widehat f(\xi)}^2 d\xi.
\eq
\end{prop}
The proof of Proposition \ref{lema-tecnico} is  rather technical, 
so that we postpone it. Assuming the result, we now give the proof of the Pohozaev identity.
\begin{proof}[Proof of Proposition~\ref{pohozaev} assuming Proposition~\ref{lema-tecnico}]
 By putting together \eqref{TW-1}--\eqref{eq-po-1} (with $\varphi=v$) and Proposition~\ref{lema-tecnico},
we have for $j\in \{1,\dots,N\}$,
\bqq
\begin{split}
\frac12 \intR {\abs{\grad v}}^2
 +\frac{1}{4} \int_{\R^N} ( W * \eta ) \eta=&
   \intR \abs{\ptl_j v}^2
-(1-\delta_{1,j})\frac c2  \intR{(v_1\ptl_1 v_2-v_2\ptl_1 v_1-\ptl_1(\chi \theta))}\\
 &+\frac{1}{4(2\pi)^N} \int_{\R^N} {\xi_j \partial_j \wh W}(\xi) \abs{\widehat \eta(\xi)}^2 d\xi,
\end{split}
\eqq
which is exactly \eqref{poh1}--\eqref{poh2}.
\end{proof}

We remark that the main problem in order to establish 
the convergence in  \eqref{limite-tecnico} is that $f$ does not  decay  fast enough at infinity.
Indeed, let us suppose that $x_jf, x_j\ptl_j f   \in L^2(\R^N)$.
Then by the dominated convergence theorem and the Plancherel identity we have
\bqq
B_n:=-\frac12 \intR  (W*f) x_j\chi_n \ptl_j f\to 
-\frac{1}{2(2\pi)^N}\intR \wh W{\overline{{\widehat{f}}}}\, \wh{x_j  \ptl_j f},\quad \text{ as } n\to\infty.
\eqq
Using  \eqref{identidad-fourier}, we conclude that
\begin{align}
\lim_{n\to\infty}B_n&=
\frac{1}{2(2\pi)^N}\intR  \wh W \abs{\wh f}^2+\frac{1}{2(2\pi)^N}\intR \wh W \xi_j \overline{\wh f}\ptl_j \wh f\label{justify}\\
&=\frac14 \intR (W*f)f-\frac{1}{4(2\pi)^N}\intR \xi_j \ptl_j \wh W\abs{\wh f}^2,\nonumber
\end{align}
where we have used the Plancherel identity, integration by parts and that  $\ptl_j \wh f\in L^{2}(\R^N)$.
This yields \eqref{limite-tecnico}, but only under these more restrictive assumptions. 
If we only have that $f\in L^2(\R^N)\cap H_{\loc}^1(\R^N)$, we can neither invoke the dominated convergence theorem
nor justify that the second integral in the r.h.s. of \eqref{justify} is finite. Therefore, to deal 
with the limit $n\to \infty$ in Proposition~\ref{lema-tecnico}, we first establish the following lemma.
\begin{lema}\label{lema-tec2}
Let $g\in L^2(\R^N)$ and $F \in L^\infty(\R^N\times \R^N)$. Assume also that 
$F(\cdot ,0)\in L^\infty(\R^N)$ and that
\bq\label{convergencia-en-cero}
F(\xi ,r_n)\to F(\xi,0), \ \text{ as }\abs{r_n}\to 0, \ \text{for a.a. }\xi\in \R^N.
\eq
For  $\varphi\in C_0^\infty(\R^N)$, we set 
 \bq\label{hyp-phi}
\widehat \varphi_n(\xi) := n^N\wh \varphi(n\xi)\ \text{ and } \ \Psi_n(\xi):=\int_{\R^N} F(\xi, r)g(\xi - r) \widehat \varphi_n(r)\,dr,
\eq
for a.a. $\xi \in \R^N$. Then 
\bq\label{claim}
 \Psi_n\to (2\pi)^NF(\cdot,0)g(\cdot)\varphi(0),\quad \text{in } L^2(\R^N), \ \text {as } n\to\infty.
\eq
\end{lema}
\begin{proof}
Let 
$$\Psi(\xi):=(2\pi)^NF(\xi,0)g(\xi)\varphi(0),\quad \text{for a.a. }\xi\in \R^N.$$
We  notice that by \eqref{hyp-phi}
\begin{equation}
\int_{\R^N} \widehat \varphi_n(r) dr =\int_{\R^N} \widehat{\varphi}(r) dr = (2\pi)^N\varphi(0),
\end{equation}
so that
$${\Psi_n}(\xi) - {\Psi}(\xi) = 
\int_{\R^N}  (F(\xi,r) g(\xi - r) - F(\xi,0)g(\xi))\widehat \varphi_n(r)  dr.$$
Then
\begin{equation}
\label{H}
\begin{split}
|{\Psi_n}(\xi) - {\Psi}(\xi)|  
\leq &\| F \|_{L^\infty(\R^{2N})} \int_{\R^{N}} |g(\xi - r) - g(\xi)| |\widehat \varphi_n(r)| dr\\
& + |g(\xi)| \int_{\R^N}  | F(\xi,r) - F(\xi,0)| |\widehat \varphi_n (r)| dr.
\end{split}
\end{equation}
On the other hand,  using \eqref{hyp-phi} and integrating by parts, we are led to
\begin{align*}
\abs{\widehat \varphi_n(\xi) }=& {{n^{N}} \bigg| \int_{\R^N} \varphi(y) e^{- i n \xi.y} dy} \bigg|\\
 =& \frac{n^{N  - 2 l}}{ |\xi|^{2 l}} \bigg| \int_{\R^N} \Delta^l \varphi(y) e^{- i n \xi.y} dy \bigg| 
 \leq \frac{n^{N  - 2 l}}{|\xi|^{2 l}} \| \Delta^l \varphi \|_{L^1(\R^N)},
\end{align*}
for any $l \in \N$ and any $\xi \neq 0$. 
Invoking this estimate for  $l = N$ and the Minkowski integral  inequality, we get
\begin{equation}
\label{lim1}
\begin{split}
 \bigg\| \int_{B(0, 1/\sqrt{n})^c} |g(\xi - r) - g(\xi)| |\widehat \varphi_n(r)| dr \bigg\|_{L^2(\R^N)} 
&  \leq  2\norm{g}_{L^2(\R^N)} \norm{\wh \varphi_n}_{L^1(B(0, 1/\sqrt{n})^c)}\\
  &\leq \frac{C(N,\varphi)}{n^{N/2}} \| g \|_{L^2(\R^N)}. 
\end{split}
\end{equation}
Similarly, we obtain
\begin{equation}
\label{lim2}
\bigg\| |g(\xi)| \int_{B(0, 1/\sqrt{n})^c}  |  F(\xi,r) - F(\xi,0)|  |\widehat \varphi_n(r)| dr \bigg\|_{L^2(\R^N)}
 \leq \frac{C(N,\varphi)}{{n^{N/2}}}  \| F \|_{L^\infty(\R^{2N})} \| g \|_{L^2(\R^N)}.
\end{equation}
On the other hand, using again the Minkowski integral inequality and \eqref{hyp-phi},
\begin{equation*}
\begin{split}
\bigg\| \int_{B(0, 1/\sqrt{n})} |g(\xi - r) - g(\xi)| &|\widehat \varphi_n(r)| dr \bigg\|_{L^2(\R^N)}\\
&\leq \Big\| \|g( \cdot-r) - g \|_{L^2(\R^N)}  |\widehat \varphi_n(r)| \Big\|_{L^1(B(0, 1/\sqrt{n}))} \\
&\leq \sup_{|y| \leq 1/\sqrt{n}} \| g(\cdot-y ) - g \|_{L^2(\R^N)}
 \|\widehat \varphi_n \|_{L^1(B(0, 1/\sqrt{n}))}  \\
&\leq  \sup_{|y| \leq 1/\sqrt{n}} \| g(\cdot-y ) - g \|_{L^2(\R^N)}\norm{\wh \varphi}_{L^1(\R^N)}.
\end{split}
\end{equation*}
Since ${g} \in L^2(\R^N)$, we know that
$$\sup_{|y| \leq h} \| g(\cdot - y) - g \|_{L^2(\R^N)} \to 0, \ {\rm as} \ h \to 0,$$
so that
\begin{equation}
\label{lim3}
\bigg\| \int_{B(0, 1/\sqrt{n})} |g(\xi - r) - g(\xi)| |\widehat \varphi_n(r)| dr \bigg\|_{L^2(\R^N)} \to 0, \ {\rm as} \ n \to + \infty.
\end{equation}
We now turn to the second term in the r.h.s. of \eqref{H}. By a change of  variables, we get that it is equal to
\bq\label{jesaispas}
|g(\xi)| \int_{B(0, \sqrt{n})}  |F(\xi,r/n) - F(\xi,0)|  |\widehat{\varphi}(r)| dr.
\eq
Since  $\wh \varphi\in L^1(\R^N)$ and 
$$ |F(\xi,r/n) - F(\xi)| |\widehat{\varphi}(r)|  \leq 2 \| F \|_{L^\infty(\R^{2N})} |\wh \varphi(r)|,$$
we can deduce from \eqref{convergencia-en-cero} and  the dominated convergence theorem that
$$\int_{B(0, \sqrt{n})} |F(\xi, r /n) - F(\xi)|  |\widehat{\varphi}(r)| dr \to 0, \ {\rm as} \ n \to + \infty,$$
for a.a. $\xi \in \R^N$. On the other hand,
\begin{align*}
 |g(\xi)| \int_{B(0, \sqrt{n})} |F(\xi,r/n) - F(\xi)| |\widehat \varphi (r)| dr\leq  2 \| F \|_{L^\infty(\R^{2N})} \| \wh \varphi \|_{L^1(\R^N)} |g(\xi)|,
\end{align*}
Therefore, again by the dominated convergence theorem,
$$\bigg\| |g(\xi)| \int_{B(0, \sqrt{n})} |F(\xi,r/n) - F(\xi)|  |\widehat{\varphi}(r)| dr \bigg\|_{L^2(\R^N)} \to 0, \ {\rm as} \ n \to + \infty.$$
By combining with \eqref{H}--\eqref{jesaispas}, we conclude \eqref{claim}, 
which finishes the proof of Lemma~\ref{lema-tec2}.
\end{proof}
\begin{proof}[Proof of Proposition \ref{lema-tecnico}]
Setting $W_m=\F^{-1}(\chi_m \wh W)=\F^{-1}(\chi_m)*W$, we have that $W_m$ is even, $W_m \in C^\infty(\R^N)$,
\bq\label{Wm-conv}
 \wh W_m\to \wh W, \ \grad \wh W_m\to \grad \wh W  \text{ a.e. and }\ W_m*g\to W*g \text{ in }L^2(\R^N),
\eq
for all $g\in L^2(\R^N)$, as ${m\to\infty}$.
Therefore
\bq\label{def-In}
 I_{n,m}:=-\frac12 \intR (W_m*{f})x_j  \chi_n\ptl_j{f}\sublim_{m\to\infty} I_n:=-\frac12 \intR  \chi_n(W*{f}){x_j}\ptl_j{f}.
\eq
Moreover, since the Fourier transform of all derivatives of $W_m$ have compact support, they are bounded in $L^2(\R^N)$.
Then, by the Plancherel theorem,  we conclude that
\bq\label{aprox-W}
W_m\in W^{k,2}(\R^N),\quad \text{for all }k\in \N.\eq
 In particular, this implies that $W_m*{f}$ belongs to $C^1(\R^N)\cap L^2(\R^N)$,
with $$\ptl_j(W_m*{f})=\ptl_j W_m*{f}.$$
Thus, integrating by parts,  we have that 
\bq\label{Inm}
I_{n,m}=P_{n,m}+Q_{n,m},
\eq where
\begin{align*}
P_{n,m}=\frac12\intR (\ptl_j W_m*{f})x_j \chi_n{f}\ \text{ and } \ Q_{n,m}=\frac12\intR (W_m*{f})(\chi_n+x_j\ptl_j \chi_n){f}.
\end{align*}
By \eqref{Wm-conv}, 
 \bq\label{Qnm}\lim_{m\to\infty}{Q_{n,m}}=\frac12 \intR (W*{f})(\chi_n+x_j\ptl_j \chi_n){f}.\eq
Since $\abs{x_j\ptl_j\chi (x)}\leq 2\norm{\chi'}_{L^\infty(\R)}$, by the dominated convergence theorem, 
\bqq
\lim_{n\to\infty} \intR (W*{f})(\chi_n+x_j\ptl_j \chi_n){f}= \intR (W*{f}){f}.
\eqq
On the other hand, by the Cauchy--Schwarz inequality, 
\bqq
\intR{\abs{\ptl_j W_m(x-y){f}(y)}dy}\leq \norm{\ptl_j W_m}_{L^2(\R^N)}\norm{{f}}_{L^2(\R^N)},  \quad x\in \R^N,
\eqq
so that 
\bq\label{fubini}
 \intR \intR \abs{\ptl_j W_m(x-y){f}(y)x_j{f}(x)\chi_n(x)}dy
dx\leq 
 2n\norm{\ptl_j W_m}_{L^2(\R^N)}\norm{{f}}^2_{L^2(\R^N)}\norm{\chi_n}_{L^2(\R^N)}.
\eq
Since $W_m$ is an even  function, $\ptl W_m$ is odd. Then, by \eqref{fubini} we can
use  the Fubini theorem to deduce that 
\bq\label{identity-Pnm}
 P_{n,m}= \frac{1}{4} \int_{\R^N} \int_{\R^N} \partial_j W_m(x - y) {f}(y) {f}(x) \big( x_j \chi_n(x) - y_j \chi_n(y) \big) dy dx,
\eq
Let us denote
\begin{equation}
\label{def:Gf} 
G_{n,m}(x):= \int_{\R^N} \partial_j W_m(x - y) {f}(y) \big( x_j \chi_n(x) - y_j \chi_n(y) \big) dy,
\end{equation}
for a.a. $x \in \R^N$. Arguing as before, using the Young inequality and \eqref{aprox-W},  we have 
\begin{align*}
\norm{G_{n,m}}_{L^1(\R^N)}&\leq \norm{\ptl_j W_m}_{L^2(\R^N)}\norm{{f}}_{L^2(\R^N)}\norm{x_j \chi_n}_{L^1(\R^N)}+
\norm{\ptl_j W_m}_{L^\infty(\R^N)}\norm{{f} x_j \chi_n}_{L^1(\R^N)},\\
\norm{G_{n,m}}_{L^2(\R^N)}&\leq \norm{\ptl_j W_m}_{L^2(\R^N)}\norm{{f}}_{L^2(\R^N)}\norm{x_j \chi_n}_{L^2(\R^N)}+
\norm{\ptl_j W_m}_{L^2(\R^N)}\norm{{f} x_j \chi_n}_{L^1(\R^N)}.
\end{align*}
Thus $G_{n,m}\in L^1(\R^N)\cap L^2(\R^N)$. Moreover, since the function $x \mapsto x_j \chi_n(x)$ is smooth on $\R^N$, we can write
$$x_j \chi_n(x) - y_j \chi_n(y) = \sum_{k = 1}^N (x_k - y_k) \theta_k(y, x - y),$$
where
$$\theta_k(y, z) :=\int_0^1 \Big( \delta_{j, k} \chi_n(y + t z) + \big( y_j + t z_j \big) \partial_k \chi_n(y + t z) \Big) dt.$$
Therefore, the function $G_{n,m}$ may be written almost everywhere as
$$G_{n,m}(x) = \sum_{k = 1}^N \int_{\R^N} (x_k - y_k) \partial_j W_m(x - y) {f}(y) \theta_k(y, x - y) dy,$$
so that its Fourier transform is equal to
\begin{align*}
\widehat G_{n,m}(p) & =  \sum_{k = 1}^N \int_{\R^N} \int_{\R^N}(x_k - y_k) \partial_j W_m(x - y) {f}(y) \theta_k(y, x - y) e^{- i p.x} dy dx\\
& =  \sum_{k = 1}^N \int_{\R^N} \int_{\R^N} z_k \partial_j W_m(z) {f}(y) \theta_k(y, z) e^{- i p.(y + z)} dy dz\\
& =  \frac{1}{(2\pi)^N}\sum_{k = 1}^N \int_{\R^N} \int_{\R^N} z_k \partial_j W_m(z) \widehat{{f}}(p - r) \tilde{\theta}_k(r, z) e^{- i p.z} dr dz,
\end{align*}
where
\begin{align*}
\tilde{\theta}_k(r, z)  :=& \int_{\R^N} \theta_k(y, z) e^{- i r.y} dy\\
 =&  \int_{\R^N} \bigg( \int_0^1 \Big( \delta_{j, k} \chi_n(y + t z) + \big( y_j + t z_j \big) \partial_k \chi_n(y + t z) \Big) dt \bigg) \theta_k(y, z) e^{- i r.y} dy\\
 =& \int_0^1 e^{i t r.z} \Big( \delta_{j, k} \widehat \chi_n(r) + \widehat{y_j \partial_k \chi_n}(r) \Big) dt.
\end{align*}
Hence, we are led to
\begin{equation*}
\wh G_{n,m}(p) =\frac{1}{(2\pi)^N}
 \sum_{k = 1}^N \int_{\R^N} \int_0^1 \widehat{z_k \partial_j W_m}(p - r t) \widehat{{f}}(p - r)
 \Big( \delta_{j, k} \widehat{\chi_n}(r) + \widehat{y_j \partial_k \chi_n}(r) \Big)\,dt\,dr.
\end{equation*}
At this stage, we note that by \eqref{Wm-conv} and \eqref{identidad-fourier},
$$\widehat{z_k \partial_j W_m}(p) \to \widehat{z_k \partial_j W}(p) =- p_j \partial_k \widehat W(p) - \delta_{k,j}\widehat W(p) \ \ {\rm a.e}, \ {\rm as} \ m \to + \infty,$$
whereas
\begin{equation*}
|\widehat{z_k \partial_j W_m}(p)| \leq \Big( 1 + 2 \| \chi' \|_{L^\infty(\R)} \Big) \| \widehat{W} \|_{L^\infty(\R^N)} + \| p_j \partial_k \widehat{W} \|_{L^\infty(\R^N)},
\end{equation*}
for a.a. $p \in \R^N$. Invoking the dominated convergence theorem, we deduce that
$$\widehat G_{n,m}(p) \to \widehat G_n(p), \quad \text{as }m \to + \infty, $$
for a.a. $p\in \R^N$, where 
$$\widehat G_n(p): = \frac{1}{(2\pi)^N}\sum_{k = 1}^N \int_{\R^N} \int_0^1 \widehat{z_k \partial_j W}(p - r t) \widehat{{f}}(p - r) 
\Big( \delta_{j, k} \widehat\chi_n(r) + \widehat{y_j \partial_k \chi_n}(r) \Big)\,dt\,dr.$$
Moreover, since
\begin{multline*}
\big| \widehat{G_{n,m}}(p) \big| \leq  \frac{1}{(2\pi)^N}\sum_{k = 1}^N \Big( \big( 1 + 2 \| \chi' \|_{L^\infty(\R)} \big) \| \widehat{W} \|_{L^\infty(\R^N)} + \| p_j \partial_j \widehat{W} \|_{L^\infty(\R^N)} \Big) \times\\
 \times \int_{\R^N} |\widehat{{f}}(p - r)| \Big| \delta_{j, k} \widehat{\chi_n}(r) + \widehat{y_j \partial_k \chi_n}(r) \Big| dr,
\end{multline*}
it follows again from the dominated convergence theorem that
$$\widehat G_{n,m} \to \widehat G_n \ {\rm in} \ L^2(\R^N), \ {\rm as} \ m \to + \infty.$$
Hence, recalling \eqref{identity-Pnm} and 
\eqref{def:Gf}, we are led to
\begin{equation}
\label{conv:K}
P_{n,m} \to P_n := \frac{1}{4} \int_{\R^N} G_n(x) {f}(x) dx, \ {\rm as} \ m \to + \infty.
\end{equation}
Finally, since 
\begin{align*}
\widehat \chi_n(p)&={n^N}\int_{\R^N} \chi_1(y)e^{- i n p.y} dy=n^N \wh \chi_1(p),\\
\widehat{y_j \partial_k \chi_n}(p)&= {n^N}\int_{\R^N} y_j \partial_k \chi_1(y)e^{- i n p.y} dy=n^N  \wh{y_j\ptl_k \chi_1}(np),
 \end{align*}
$\chi_1=1$ and $\partial_k \chi_1=0$ on $B(0,1)$, applying Lemma~\ref{lema-tec2} 
with
\bqq
\varphi=\delta_{j,k}\chi_1+{y_j \partial_k \chi_1},\quad F(p,r)=\int_0^1 \widehat{z_k \partial_j W}(p - r t)\,dt,\quad g=\wh f,  \
\eqq
we conclude that 
\bq\label{G_}
\wh G_n\to \widehat{z_j \partial_j W} \widehat{{f}} \quad \text{in } L^2(\R^N), \quad \text{as }n\to\infty.
\eq
Therefore, in view of \eqref{conv:K}, \eqref{G_} and the Plancherel identity, we have
$$P_n \to \frac{1}{4(2\pi)^N} \int_{\R^N} \widehat{z_j \partial_j W}(p) |\wh {f}(p)|^2 dp, \ {\rm as} \ n \to + \infty.$$
By combining with \eqref{identidad-fourier}, \eqref{def-In}, \eqref{Inm}, \eqref{Qnm} and \eqref{conv:K}, we obtain \eqref{limite-tecnico}.
\end{proof}
\section{Proof of the main results}\label{main}
We are now in position to provide the proofs of the results stated in Subsection~\ref{statement}.
\begin{proof}[Proof of Theorem \ref{teorema}]
For $j\in \{1,\dots,N\}$, let us  introduce the notation
\begin{gather*}
\K_j:=\frac12\intR \abs{\ptl_j v}^2, \quad \K:=\sum_{j=1}^N{\K_j}, \quad \RR_j:=\frac1{4(2\pi)^N} \intR \xi_j\ptl_j \wh W \abs{\wh \eta}^2, \\
      \PP:=\intR{(v_1\ptl_1 v_2-v_2\ptl_1 v_1-\ptl_1(\chi \theta))}, \quad \U:=\frac14 \intR (W*\eta)\eta.  
 \end{gather*}  
In this way  
 \bq\label{k-u}
 E(v)=\K+\U
 \eq
 and Propositions~\ref{prop-fourier} and \ref{pohozaev} read 
\begin{align}
\K +2\U&=-\frac{c\,\ell_{c}}{2(1+\ell_{c})}\PP\label{dem-1},\\
\K+\U&=2\K_1+\RR_1,\label{dem-2}\\
\K+\U&=2\K_j-\frac{c}2\PP+\RR_j,\label{dem-3} 
\end{align}
for all $j\in \{2,\dots,N\}$.
From \eqref{dem-1} and \eqref{dem-3}, we obtain
\bq\label{dem-4}
(1+2\ell_{c})\K_j+\sum_{\substack{k=1\\ k\neq j}}^N \K_k+(\ell_{c}+2)\U=-\ell_{c}\RR_j,  \quad j\in \{2,\dots,N\}.
\eq
Therefore, we can write \eqref{k-u}, \eqref{dem-2} and \eqref{dem-4} as the linear system $Az=b$, with 
$$z=(\K_1,\K_2,\dots, \K_n,\U),\quad b=(\RR_1,-\ell_{c}\RR_2,\dots, -\ell_{c}\RR_N,E(v))$$
 and $A\in \R^{N+1\times N+1}$ given by
\bqq
A_{i,j}=\begin{cases}
         -1, &\textup{ if } i=j=1,\\
         2+\ell_{c}, &\textup{ if } j=N+1,\ 1<i<N+1,\\
1+2\ell_{c}, &\textup{ if } i=j, \ i\neq 1,\\
1, &\textup{otherwise}.
        \end{cases}
\eqq
Let $\sigma=(\sigma_1,\sigma_2,\dots,\sigma_N,-1)$. If $K(v)=0$, $v$ is constant. 
 Therefore we suppose that $K(v)>0$.  Then  using \eqref{sigma-1}, 
\bq\label{b}
  \begin{split}
  b^T \sigma &=\sigma_1\RR_1-\ell_c\sum_{k=2}^N  \sigma_k \RR_k - E(v)\\
  &=\frac1{4(2\pi)^N} \intR \abs{\wh \eta(\xi)}^2 \left( \sigma_1\xi_1 \ptl_1\wh W(\xi)-\ell_c 
  \sum_{k=2}^N \sigma_k \xi_k \ptl_k\wh W(\xi) - \wh W(\xi) \right)d\xi- K(v)\\
&\leq - K(v)<0.
 \end{split}
\eq
On the other hand, 
\bqq
(A^T\sigma)_j=\begin{cases}
                   -\sigma_1+\sum_{k=2}^N\sigma_k-1,\quad &\text{if }j=1,\\
                   \sigma_1+\sum_{k=2}^N\sigma_k+2\ell_c \sigma_j -1,\quad &\text{if }2\leq j\leq N,\\
                   \sigma_1+(\ell_c+2)\sum_{k=2}^N\sigma_k-1,\quad &\text{if }2\leq j= N+1.\\
                  \end{cases}
\eqq
Consequently, by \eqref{sigma-2},  $A^T \sigma\geq 0$.
However, since $z\geq 0$, this inequality together with  \eqref{b}
contradict Farkas' Lemma.
\end{proof}
\begin{proof}[Proof of Theorem \ref{teo1}]
It is an immediate consequence of Theorem \ref{teorema} and Lemma \ref{curva}.
\end{proof}
\begin{proof}[Proof of Theorem \ref{teo-c-0}]
Using the notation of the proof of Theorem~\ref{teorema}, by \eqref{der-c-0}
and Proposition \ref{pohozaev} we conclude that
\bqq
\K+\U\leq 2\K_j,\quad \text{ for all }j\in\{1,\dots,N\}.
\eqq 
Thus, summing over $j$,
 \bq\label{dem-c-0}
 \U \leq\frac{2-N}{N}\K.
 \eq
Since $N\geq 2$, $\K\geq 0$ and $\U\geq 0$, inequality
\eqref{dem-c-0} implies that $\U=0$ and therefore $v$ 
is constant.
\end{proof}

\begin{proof}[Proof of Corollary \ref{cor-teo1}]
Let us take $\sigma_1=-1$ and $\bar \sigma:=\sigma_2=\dots=\sigma_N>0$.
In order to fulfill \eqref{sigma-2-2}, we finally fix
\bqq
\bar \sigma=\max\left\{\frac{2}{(N-1)(\alpha_c+2)}, \frac{2}{N-1+\alpha_c}\right\}.
\eqq
Then $\alpha_c \bar \sigma \leq \max\{1,2/(N-1)\}$, so that 
\bqq
\wh W(\xi)+\alpha_c  \sum_{k=2}^N\sigma_k  \abs{\xi_k\ptl_k \wh W(\xi)} 
\geq 
 \wh W(\xi)-\max\left\{1,\frac2{N-1}\right\} \sum_{k=2}^N \abs{\xi_k\ptl_k \wh W(\xi)}- \abs{\xi_1\ptl_1 \wh W(\xi)}.
\eqq
Therefore the conclusion follows from \eqref{hyp-cor} and Theorem \ref{teo1}.
\end{proof}
\begin{proof}[Proof of Corollary \ref{cor-teo}]
Taking  $\sigma_1=0$ and $\sigma_2=\dots=\sigma_N=1/(N-1)$, we have that \eqref{sigma-2-2}
is satisfied. Let $$m:=\inf_{\xi\in \R^N} \frac{(N-1)\wh W(\xi)}{\sum_{k=2}^N \abs{\xi_k\ptl_k \wh W(\xi)}}.$$
If $m=+\infty$,  $\xi_j\ptl_j \wh W(\xi)=0$ a.e. for all $j\in \{2,\dots,N\}$ and then \eqref{sigma-1-1}
is fulfilled. 
If $m<\infty$, we note that \eqref{ine:cor} 
implies $\alpha_c \leq m$, so that 
\bqq
\wh W(\xi)+\alpha_c  \sum_{k=2}^N\sigma_k  \abs{\xi_k\ptl_k \wh W(\xi)} 
\geq 
 \frac{m-\alpha_c}{N-1}\sum_{k=2}^N \abs{\xi_k\ptl_k \wh W(\xi)} \geq 0.
\eqq
Then Theorem~\ref{teo1} yields the conclusion.
\end{proof}
\begin{proof}[Proof of Corollary \ref{ultimo-cor}]
The proof is analogous to that of Corollaries \ref{cor-teo1} and \ref{cor-teo}. 
The only difference is that we invoke  Theorem~\ref{teorema} instead of Theorem~\ref{teo1} to conclude.
\end{proof}
 \bibliographystyle{abbrv}   
\bibliography{ref}
\end{document}